\documentclass[12pt]{article}
\usepackage{amssymb,amsmath,amsthm,psfrag,graphics,latexsym,graphicx,tikz}
\usepackage{caption,psfrag}
\usepackage{subcaption}
\usetikzlibrary{arrows, decorations.markings}

\numberwithin{equation}{section}
\newtheorem{theorem}{Theorem}[section]
\newtheorem{lemma}[theorem]{Lemma}
\newtheorem{proposition}[theorem]{Proposition}

\newtheorem{problem}[theorem]{Problem}

\title{Cycles with two blocks in $k$-chromatic digraphs}

\author{
Ringi Kim\footnote{Department of Combinatorics and Optimization, University of Waterloo, Waterloo, Canada. E-mail: ringikim2@gmail.com.}
~~~~~
Seog-Jin Kim\footnote{Department of Mathematics Education, Konkuk University, Seoul, Republic of Korea. E-mail: skim12@konkuk.ac.kr. This research was supported by Basic Science Research Program through the National Research Foundation of Korea(NRF) funded by the Ministry of Education(NRF-2015R1D1A1A01057008).}
~~~~~
Jie Ma\footnote{School of Mathematical Science,
University of Science and Technology of China, Hefei,
P.R. China. E-mail: jiema@ustc.edu.cn. Research partially supported by NSFC projects 11501539 and 11622110.}
~~~~~
Boram Park\footnote{Department of mathematics, Ajou University, Suwon, Republic of Korea. E-mail: borampark@ajou.ac.kr.
This research was supported by Basic Science Research Program through the National Research Foundation of Korea (NRF) funded by the Ministry of  Science, ICT \& Future Planning (NRF-2015R1C1A1A01053495).}
}

\begin{document}

\maketitle

\begin{abstract}
Let $k$ and $\ell$ be positive integers. A {\it cycle with two blocks} $c(k,\ell)$ is an oriented cycle
which consists of two internally (vertex) disjoint directed paths of lengths at least $k$ and $\ell$,
respectively, from a vertex to another one.
A problem of Addario-Berry, Havet and Thomass\'e \cite{AHT} asked if, given positive integers $k$ and $\ell$ such that $k+\ell\ge 4$,
any strongly connected digraph $D$ containing no $c(k,\ell)$ has chromatic number at most $k+\ell-1$.
In this paper, we show that such digraph $D$ has chromatic number at most $O((k+\ell)^2)$,
improving the previous upper bound $O((k+\ell)^4)$ of \cite{CHLN}. In fact, we are able to find a digraph
which shows that the answer to the above problem of \cite{AHT} is no.
We also show that if in addition $D$ is Hamiltonian, then its underlying simple graph is $(k+\ell-1)$-degenerate and thus the chromatic number of $D$ is at most $k+\ell$, which is tight.
\end{abstract}


\bigskip


\section{Introduction}
Throughout this paper, all graphs $G$ and digraphs $D$ are {\it simple},
that is, there are no loops and no multiple edges in $G$ or $D$ (though a pair of opposite arcs is allowed in $D$).
Unless otherwise specified, by a path, a walk or a cycle in a digraph $D$ we {\it always} mean a directed one.
The \textit{length} $|P|$ of a walk $P$ is the number of arcs it contains.
An \textit{orientation} of a graph $G$ is a digraph obtained by giving a direction to each edge of $G$,
and the \textit{underlying graph} of a digraph $D$ is the simple graph obtained by ignoring the directions of the arcs. 
The chromatic number $\chi(D)$ of a digraph $D$ is the chromatic number of its underlying graph.
And we say $D$ is {\it $n$-chromatic} if $\chi(D)=n$.

A classic theorem of Gallai and Roy \cite{Gallai,Roy} says that any $n$-chromatic digraph contains a path with $n$ vertices.
This motivates the study of {\it $n$-universal} digraphs, that is, digraphs contained in any $n$-chromatic digraphs.
It is known that $n$-universal digraphs must be oriented trees (an orientation of an undirected tree), and
Burr \cite{Burr} conjectured that every oriented tree of $n$ vertices is $(2n-2)$-universal,
which is remained to be open (for more information see \cite{EI04-2,HT91,HT00,HT00-2}).
For positive integers $k$ and $\ell$, a {\it path with two blocks} $P(k,\ell)$ is an orientation of the undirected path of $k + \ell + 1$ vertices having two maximal paths, that is, either starting with $k$ forward arcs followed by $\ell$ backward arcs, or starting with $k$ backward arcs followed by $\ell$ forward arcs.
In \cite{EI04}, El-Sahili conjectured that every path of $n\ge 4$ vertices with two blocks is $n$-universal.
El-Sahili and Kouider \cite{EK} proved that every such path is $(n+1)$-universal,
and then Addario-Berry, Havet, and Thomass\'e \cite{AHT} confirmed the conjecture.

For positive integers $k$ and $\ell$, a {\it cycle with two blocks} $c(k,\ell)$ (or we call it a {\it 2-block cycle}) is a digraph obtained by an orientation of an undirected cycle, which consists of two internally (vertex) disjoint paths of lengths at least $k$ and $\ell$, respectively, from a vertex to another one.
A natural question is to ask if a digraph with high chromatic number can contain a 2-block cycle $c(k,\ell)$.
In \cite{BW}, Benhocine and Wojda proved that every tournament of $n\ge 4$ vertices
contains a 2-block cycle $c(k,\ell)$ for any positive integers $k,\ell$ with $k+\ell=n$.
However, for general digraphs the answer is no: as shown by Gy\'arf\'as and Thomassen (see \cite{AHT}), there exist digraphs
with arbitrary large chromatic number which contain no cycles with two blocks.
A digraph $D$ is \textit{strongly connected} (or \textit{strong}, for short) if for any two vertices
$u$ and $v$ of $D$, there are a path from $u$ to $v$ and a path from $v$ to $u$.
The authors of \cite{AHT} noticed that the digraphs found by Gy\'arf\'as and Thomassen are not strongly connected,
and they proposed the following problem on cycles with two blocks for strongly connected digraphs.
\begin{problem}[Addario-Berry, Havet, and Thomass\'e, \cite{AHT}]\label{prob}
Let $D$ be an $n$-chromatic strongly connected digraph, $n\ge4$, and
let $k$ and $\ell$ be positive integers such that $k +\ell = n$.
Does such $D$ contain a 2-block cycle $c(k,\ell)$?
\end{problem}

This problem also can be viewed as an extension of the following classic theorem of Bondy \cite{Bondy},
which asserts the statement for cycles. (A cycle can be considered as a 2-block cycle $c(k,0)$.)
\begin{theorem}[Bondy, \cite{Bondy}]\label{thm:bondy}
Any strong digraph $D$ contains a cycle of length at least $\chi(D)$.
\end{theorem}

We mention that a different extension of Bondy's theorem was obtained in \cite{CMZ}.
Very recently, among other results, Cohen, Havet, Lochet, and Nisse \cite{CHLN} firstly obtained a finite upper bound of $\chi(D)$
for strong digraphs $D$ containing no $c(k,\ell)$. Precisely, they proved the following:

\begin{theorem}[Cohen, Havet, Lochet, and Nisse, \cite{CHLN}]\label{thm:O(k4)}
Let $k$ and $\ell$ be integers such that $k\ge \ell \ge 2$ and $k\ge4$, and
 $D$  a strong digraph with no 2-block cycle $c(k,\ell)$.
Then
\[\chi(D) \le (k +\ell-2)(k +\ell-3)(2\ell+2)(k +\ell+1).\]
\end{theorem}

In this paper, we improve the above upper bound $O((k+\ell)^4)$ to $O((k+\ell)^2)$
by using a quite different approach from \cite{CHLN}.
The following is our main result.

\begin{theorem}\label{thm:main:general}
Let $k$ and $\ell$ be integers such that $k\ge \ell\ge 1$ and $k\ge 2$, and
 $D$  a strong digraph with no 2-block cycle $c(k,\ell)$.
Then
  \[\chi(D) \le 2(2k-3)(k+2\ell-1)< 12k^2. \]
\end{theorem}

As a key step and a result of independent interest,
we consider Hamiltonian digraphs and obtain the following tight result.
A digraph $D$ is {\it Hamiltonian}, if it contains a {\it Hamiltonian cycle}, that is, a cycle passing through all vertices of $D$. For a positive integer $d$, a digraph $D$ is {\it $d$-degenerate}, if its underlying graph $G$ is $d$-degenerate,
that is, any subgraph of $G$ contains a vertex having at most $d$ neighbors (or having degree at most $d$) in it.
Note that if a graph $G$ is $d$-generate, then it is $(d+1)$-colorable.

\begin{theorem}\label{thm:Hamiltonian}
Let $k$ and $\ell$ be positive integers such that $k+\ell\ge 3$, and $D$ a Hamiltonian digraph with no 2-block cycle $c(k,\ell)$.
Then $D$ is $(k+\ell-1)$-degenerate, which implies that $\chi(D)\le k+\ell$.
\end{theorem}

When $k=\ell=1$, if a Hamiltonian digraph $D$ has no 2-block cycle $c(k,\ell)$, then $D$ is an induced cycle and so  $\chi(D)\le 3=k+\ell+1$ and the equality holds only when $D$ is an odd cycle.
To see the tightness of Theorem~\ref{thm:Hamiltonian}, we consider the strong tournament $T$ in Figure~\ref{T},
which contains no $c(4,1)$ and has $\chi(T)=5$.
We remark that this tournament $T$ answers Problem~\ref{prob} negatively (for $k=4$ and $\ell=1$).

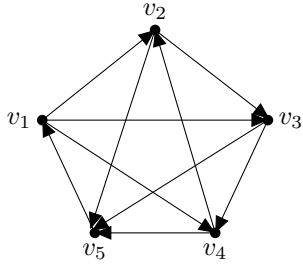
\begin{figure}
\centering
 \begin{tikzpicture} [>=triangle 45,mydeco/.style = {decoration = {markings,mark = at position #1 with {\arrow{<}}}} ]
     \path (-1.5,0) coordinate (1) (0,1.2) coordinate (2)
     (1.5,0) coordinate (3) (0.8,-1.5) coordinate (4)
     (-0.8,-1.5) coordinate (5);
     \draw[->] (1) to (2); \draw[->] (2) to (1.52,0.02);\draw[->] (3) to (4);\draw[->] (5) to (1);\draw[->] (4) to (5); \draw[->] (4) to (2); \draw[->] (1) to (3); \draw[->] (1) to (4); \draw[->] (2) to (-0.86,-1.45); \draw[->] (3) to (-0.84,-1.47);
       \fill (1) node[left]{\footnotesize$v_1$}
             (2) node[above]{\footnotesize$v_{2}$}
             (3) node[right]{\footnotesize$v_{3}$}
             (4) node[below]{\footnotesize$v_{4}$}
             (5)  node[below]{\footnotesize$v_{5}$};
        \fill (1) circle (2pt)
              (2) circle (2pt)  (4) circle (2pt)
              (3) circle (2pt) (5) circle (2pt);
              \end{tikzpicture}
\caption{A tight example for Theorem~\ref{thm:Hamiltonian}}\label{T}\end{figure}

We introduce some basic notations and terminologies.
Let $D$ be a digraph. For two walks $P$ and $Q$ of $D$, if the terminal vertex
of $P$ and the starting vertex of $Q$ are the same, then we denote by $P+Q$  the walk through $P$ and then $Q$.
For a cycle $C$ of $D$ and any two vertices $u$ and $v$ on $C$, we denote by $uCv$ the subpath of $C$ from $u$ to $v$ along $C$.
For $S\subseteq V(D)$, we denote by $D[S]$ the induced subdigraph of $D$ on the vertex set $S$.
Let $G$ be a graph.
For a vertex $u$ of $G$, the {\it neighborhood} $N_G(u)$ of $u$ contains all neighbors of $u$ in $G$, and
the {\it closed neighborhood} $N_G[u]$ of $u$ is defined by $N_G[u]:=N_G(u)\cup \{u\}$.
In addition, we denote by
$ \delta(G) =\min_{u\in V(G)} |N_G(u)|$.
For a positive integer $k$, denote $[k]:=\{1,2,...,k\}$.

The paper is organized as follows. In Section 2, we prove Theorem~\ref{thm:Hamiltonian} for Hamiltonian digraphs.
And in Section 3, we complete the proof of Theorem~\ref{thm:main:general} for general strong digraphs.

\section{Hamiltonian Digraphs: Proof of Theorem~\ref{thm:Hamiltonian}}
We devote this section to prove Theorem~\ref{thm:Hamiltonian}.
Throughout this section, if $D$ is a digraph, $G$ is its underlying graph, and $S\subset A(D)$, then we allow, with slight abuse of notation, to denote by $E(S)$ for the set of edges of $G$ obtained by ignoring the directions of arcs of $S$.
We begin with the following useful lemma, which will be iteratively applied later.

\begin{lemma}\label{lem:crossing}
Let  $D$ be a digraph which has a Hamiltonian cycle $C$, and $G$ its underlying graph.
Suppose that $u, v, x, y$ are four distinct vertices such that $uv,xy\in E(G)\setminus E(C)$ such that $x\in V(uCv)$ and $y\in V(vCu)$ (see Figure~\ref{fig-lem:crossing}).
For positive integers $k$ and $\ell$,
if $|uCx|\ge k-1$ and $|vCy|\ge \ell-1$, then $D$ contains a 2-block cycle $c(k,\ell)$, unless one of the following occurs:
\begin{itemize}
\item[(a)] $|uCx|=k-1$ and $(u,v), (y,x)\in A(D)$, or
\item[(b)] $|vCy|=\ell-1$ and $(v,u), (x,y)\in A(D)$.
\end{itemize}
\end{lemma}

\begin{figure}[h!]
\begin{subfigure}{.25\textwidth}\centering
    \begin{tikzpicture} [scale=0.65, >=triangle 45,mydeco/.style = {decoration = {markings,mark = at position #1 with {\arrow{<}}}} ]
     \path (0:0) coordinate (0)
           (0:2cm) coordinate (y) (-90:2cm) coordinate (v)
           (100:2cm) coordinate (u) (150:2cm) coordinate (x);
     \draw (0) circle (2cm);
     \path (x)edge (y); \path (u)edge (v);
       \fill (x) node[left]{\footnotesize$u$}
             (y) node[right]{\footnotesize$v$}
             (u) node[above]{\footnotesize$x$}
             (v)  node[below]{\footnotesize$y$};
        \fill (v) circle (3pt)
              (x) circle (3pt)
              (y) circle (3pt)
              (u) circle (3pt);
    \draw[postaction = {mydeco=-0.0416 ,decorate},
           postaction = {mydeco=0.5 ,decorate}, postaction = {mydeco=0.35 ,decorate}, postaction = {mydeco=0.2 ,decorate}] (0:0) circle (2cm);
\draw[-]  (u) to [bend right=75] node[left] {\tiny length$\ge k-1$} (x);
\draw[-]  (y) to [bend left=95] node[right] {\tiny length$\ge \ell-1$} (v);
  \fill (-90:3cm) node[below]{};
\end{tikzpicture}
    \end{subfigure} \qquad    \quad
    \begin{subfigure}{.25\textwidth}\centering
    \begin{tikzpicture} [scale=0.65, >=triangle 45,mydeco/.style = {decoration = {markings,mark = at position #1 with {\arrow{<}}}} ]
     \path (0:0) coordinate (0)
           (0:2cm) coordinate (y) (-90:2cm) coordinate (v)
           (100:2cm) coordinate (u) (150:2cm) coordinate (x);
     \draw (0) circle (2cm);
     \draw[->] (x) to (y); \draw[->] (v) to (u);
       \fill (x) node[left]{\footnotesize$u$}
             (y) node[right]{\footnotesize$v$}
             (u) node[above]{\footnotesize$x$}
             (v)  node[below]{\footnotesize$y$};
        \fill (v) circle (3pt)
              (x) circle (3pt)
              (y) circle (3pt)
              (u) circle (3pt);
    \draw[postaction = {mydeco=-0.0416 ,decorate},
           postaction = {mydeco=0.5 ,decorate}, postaction = {mydeco=0.35 ,decorate}, postaction = {mydeco=0.2 ,decorate}] (0:0) circle (2cm);
\draw[-]  (u) to [bend right=75] node[left] {\tiny length$=k-1$} (x);
\draw[-]  (y) to [bend left=95] node[right] {\tiny length$\ge \ell-1$} (v); \fill (-90:3cm) node[below]{(a)};
\end{tikzpicture}
    \end{subfigure}   \qquad    \quad
    \begin{subfigure}{.25 \textwidth}\centering
    \begin{tikzpicture} [scale=0.65, >=triangle 45,mydeco/.style = {decoration = {markings,mark = at position #1 with {\arrow{<}}}} ]
     \path (0:0) coordinate (0)
           (0:2cm) coordinate (y) (-90:2cm) coordinate (v)
           (100:2cm) coordinate (u) (150:2cm) coordinate (x);
     \draw (0) circle (2cm);
\draw[->] (y) to (x); \draw[->] (u) to (v);
       \fill (x) node[left]{\footnotesize$u$}
             (y) node[right]{\footnotesize$v$}
             (u) node[above]{\footnotesize$x$}
             (v)  node[below]{\footnotesize$y$};
        \fill (v) circle (3pt)
              (x) circle (3pt)
              (y) circle (3pt)
              (u) circle (3pt);
    \draw[postaction = {mydeco=-0.0416 ,decorate},
           postaction = {mydeco=0.5 ,decorate}, postaction = {mydeco=0.35 ,decorate}, postaction = {mydeco=0.2 ,decorate}] (0:0) circle (2cm);
\draw[-]  (u) to [bend right=75] node[left] {\tiny length$\ge k-1$} (x);
\draw[-]  (y) to [bend left=95] node[right] {\tiny length$=\ell-1$} (v);
\fill (-90:3cm) node[below]{(b)};
\end{tikzpicture}  \end{subfigure}
\caption{Figures for Lemma~\ref{lem:crossing}}\label{fig-lem:crossing}\end{figure}
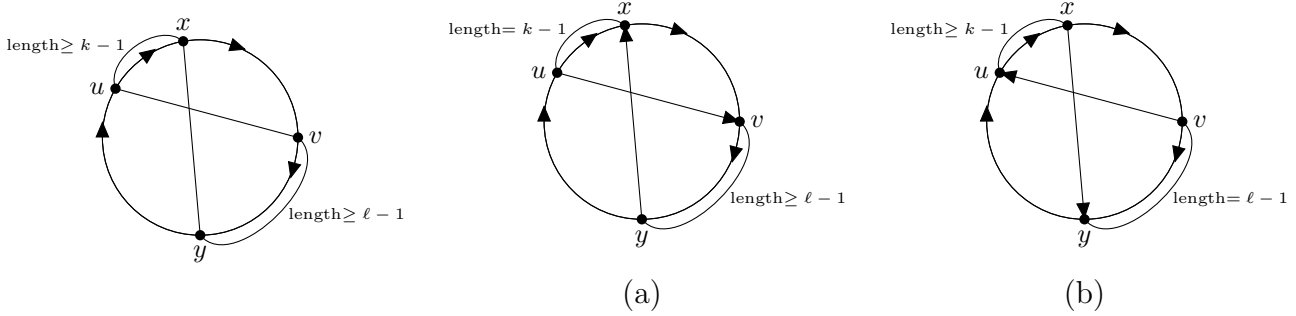

\begin{proof}
Note that there are four difference cases of the directions of the edges $uv$ and $xy$.
If $(u,v), (x,y)\in A(D)$,
then $uCx+(x,y)$ and $(u,v)+vCy$ are internally disjoint paths from $u$ to $y$ of length  $|uCx|+1$ and $|vCy|+1$, respectively.
If $(u,v), (y,x)\in A(D)$,
then $uCx$ and $(u,v)+vCy+(y,x)$ are internally disjoint paths from $u$ to $x$ of length  $|uCx|$ and $|vCy|+2$, respectively.
If $(v,u), (x,y)\in A(D)$,
then $(v,u)+uCx+(x,y)$ and $vCy$ are internally disjoint paths from $v$ to $y$ of length  $|uCx|+2$ and $|vCy|$, respectively.
If $(v,u), (y,x)\in A(D)$,
then $(v,u)+uCx$ and $vCy+(y,x)$ are internally disjoint paths from $v$ to $x$ of length $|uCx|+1$ and $|vCy|+1$, respectively.
Now it is easy to verify the conclusion, under the above observations.\end{proof}

The following lemma will be essential for Theorem~\ref{thm:Hamiltonian}.

\begin{lemma}\label{lem:Ham-delta}
Let $k$ and $\ell$ be positive integers such that $k+\ell\ge 3$.
Let $D$ be a Hamiltonian digraph and $G$ its underlying  graph.
If $\delta(G)\ge k+\ell$, then $D$ contains a 2-block cycle $c(k,\ell)$.
\end{lemma}

\begin{proof}
Suppose for a contradiction that $D$ has no 2-block cycle $c(k,\ell)$. Let $C=v_0, v_1,..., v_{n-1}, v_0$ be a Hamiltonian cycle of $D$, where $n=|V(D)|\ge k+\ell+1$.


\bigskip

{\bf Claim 1:} There is no $r$ such that either $v_r v_{r+k+1} \in E(G)$ or $v_r v_{r+\ell+1} \in E(G)$.

\begin{proof}[Proof of Claim 1.]
Suppose that such $r$ exists. By symmetry between $k$ and $\ell$, we may assume that $v_0 v_{k+1} \in E(G)$.
Note that $v_0$, $v_1$, $\ldots$, $v_{k+\ell}$  are distinct vertices in $G$, since $k+\ell +1 \le |V(G)|$.
Thus  $v_0 v_{k+1}$ is an edge not in $E(C)$ if $|V(G)|=n>k+2$.
To reach the final contradiction, we prove a series of assertions (1.1) $\sim$ (1.4) as follows (see Figure~\ref{fig:claim1} for illustration).

\medskip

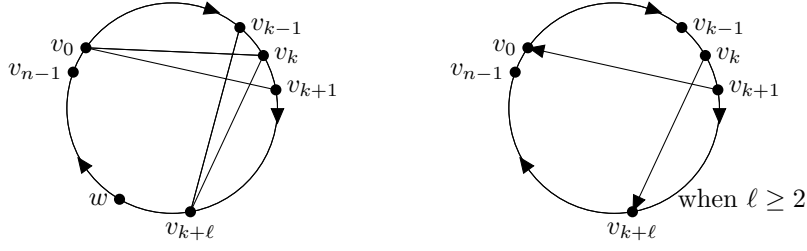
\begin{figure}\centering
\begin{subfigure}{.25\textwidth}\centering
 \begin{tikzpicture} [scale=0.7, >=triangle 45,mydeco/.style = {decoration = {markings,mark = at position #1 with {\arrow{<}}}} ]
     \path (0:0) coordinate (0)  (160:2cm) coordinate (vn)
           (145:2cm) coordinate (v0)  (50:2cm) coordinate (vk-)(30:2cm) coordinate (vk)
           (10:2cm) coordinate (vk+)
           (240:2cm) coordinate (w)
           (280:2cm) coordinate (vkl);
     \draw (0) circle (2cm);
     \path (v0)edge (vk); \path (vkl)edge(vk-);
     \path (vk+)edge(v0);\path (vk)edge(vkl);
    \path (v0)edge (vk);\path (vk-)edge (vkl);
       \fill (v0) node[left]{\footnotesize$v_0$}
             (vn) node[left]{\footnotesize$v_{n-1}$}
             (vk-) node[right]{\footnotesize$v_{k-1}$}
             (w) node[left]{\footnotesize$w$}
             (vk) node[right]{\footnotesize$v_{k}$}(vkl) node[below]{\footnotesize$v_{k+\ell}$}
             (vk+)  node[right]{\footnotesize$v_{k+1}$};
        \fill (v0) circle (3pt)  (w) circle (3pt)
              (vn) circle (3pt)  (vk) circle (3pt)
              (vk-) circle (3pt) (vkl) circle (3pt)
              (vk+) circle (3pt);
    \draw[postaction = {mydeco=-0.4 ,decorate},
             postaction = {mydeco=0 ,decorate}, postaction = {mydeco=0.2 ,decorate}] (0:0) circle (2cm);
\end{tikzpicture}
\end{subfigure} \qquad    \quad
\begin{subfigure}{.25\textwidth}\centering
 \begin{tikzpicture} [scale=0.7, >=triangle 45,mydeco/.style = {decoration = {markings,mark = at position #1 with {\arrow{<}}}} ]
     \path (0:0) coordinate (0)  (160:2cm) coordinate (vn) (300:2cm) coordinate(l)
           (145:2cm) coordinate (v0)  (50:2cm) coordinate (vk-)(30:2cm) coordinate (vk)
           (10:2cm) coordinate (vk+)
           (280:2cm) coordinate (vkl);
     \draw (0) circle (2cm);
     \draw[->] (vk+) to (v0); \draw[->] (vk) to (vkl);
       \fill (v0) node[left]{\footnotesize$v_0$}
             (vn) node[left]{\footnotesize$v_{n-1}$}
             (l) node[right]{\footnotesize$\text{when } \ell\ge 2$}
             (vk-) node[right]{\footnotesize$v_{k-1}$}
             (vk) node[right]{\footnotesize$v_{k}$}(vkl) node[below]{\footnotesize$v_{k+\ell}$}
             (vk+)  node[right]{\footnotesize$v_{k+1}$};
        \fill (v0) circle (3pt)
              (vn) circle (3pt)  (vk) circle (3pt)
              (vk-) circle (3pt) (vkl) circle (3pt)
              (vk+) circle (3pt);
    \draw[postaction = {mydeco=-0.4 ,decorate},
             postaction = {mydeco=0 ,decorate}, postaction = {mydeco=0.2 ,decorate}] (0:0) circle (2cm);
\end{tikzpicture}
\end{subfigure}
\caption{Figures for Claim 1}\label{fig:claim1}\end{figure}
(1.1) $N_G[v_k]=V(v_0Cv_{k+\ell})$.

\medskip

\noindent Suppose not, then in view of $|N_G(v_k)|\ge k+\ell$, $v_k$ has a neighbor $w\in V(v_{k+\ell+1}Cv_{n-1})$ in $G$. Then $k+\ell+1 \le n-1$ or $k+2<n$. Thus the edges $v_0v_{k+1}$ and $v_{k}w$ are in $G$ but not in $E(C)$. Since $|v_0Cv_k|\ge k$ and $|v_{k+1}Cw|\ge \ell$,
Lemma~\ref{lem:crossing} then forces a 2-block cycle $c(k,\ell)$ in $D$. This contradiction completes the proof of (1.1).

\medskip

(1.2) If $\ell\ge 2$, then $(v_{k+1},v_0),(v_k,v_{k+\ell})\in A(D)$.

\medskip

\noindent Suppose that $\ell\ge 2$.
Then $v_kv_{k+\ell} \in E(G)\setminus E(C)$.
Moreover, $n>k+2$ and so $v_0v_{k+1}$ is not an edge of $E(C)$.
By considering the pair of edges $v_0v_{k+1}$ and $v_kv_{k+\ell}$ of $G$ not in $E(C)$,
since $|v_0Cv_k|\ge k$ and $|v_{k+1Cv_{k+\ell}}|\ge \ell-1$,  Lemma~\ref{lem:crossing} shows that, to avoid $c(k,\ell)$, the orientation of these edges in $D$ are $(v_{k+1},v_0),(v_k,v_{k+\ell})\in A(D)$.

\medskip

(1.3) $k\ge 2$ and thus, $v_{k-1}$ is a vertex distinct from $v_0$.

\medskip

\noindent Otherwise, $k=1$ and so $\ell\ge 2$, implying that $(v_k,v_{k+\ell})\in A(D)$ by (1.2);
then the arc $(v_k,v_{k+\ell})$ and the path $v_kCv_{k+\ell}$ together form a 2-block cycle $c(k,\ell)$  (where $k=1$), a contradiction.

\medskip

(1.4) $N_G[v_{k-1}]=V(v_0 C v_{k+\ell})$.

\medskip

\noindent Suppose not, then in view of $|N_G(v_{k-1})|\ge k+\ell$,
$v_{k-1}$ has a neighbor $w\in V(v_{k+\ell+1}Cv_{n-1})$ in $G$.
Then $k+\ell+1\le n-1$ or $k+2<n$, and so $v_0v_{k+1}\in E(G)\setminus E(C)$.
By (1.3), as $k\ge 2$, $v_{k-1}\neq v_0$, and so $v_{k-1}w\in E(G)\setminus E(C)$.
By considering two edges $v_0v_{k+1}, v_{k-1}w\in E(G)\setminus E(C)$,
since $|v_0Cv_{k-1}|=k-1$ and $|v_{k+1}Cw|\ge \ell$,
by Lemma~\ref{lem:crossing}, we get that $(v_0,v_{k+1}), (w,v_{k-1})\in A(D)$.
We can not have $\ell\ge 2$, as otherwise it would contradict (1.2).
Thus, $\ell=1$, then the path $v_0Cv_{k+1}$ and the arc $(v_0,v_{k+1})$ form a 2-block cycle $c(k,\ell)$ (where $\ell=1$), a contradiction.
This proves (1.4).

\medskip

As $k\ge 2$, we now observe that $v_0v_k$ and $v_{k-1}v_{k+\ell}$ are edges of $E(G)\setminus E(C)$.
Since $|v_0Cv_{k-1}|=k-1$ and $|v_kCv_{k+\ell}|=\ell$,
applying Lemma~\ref{lem:crossing}, we obtain that $(v_0,v_{k}), (v_{k+\ell},v_{k-1})\in A(D)$.
If $(v_{k+1},v_0)\in A(D)$, then the paths $(v_{k+1},v_0)+v_0Cv_{k-1}$ and $v_{k+1}Cv_{k+\ell}+(v_{k+\ell},v_{k-1})$
form a 2-block cycle $c(k,\ell)$ of $D$, a contradiction.
Hence, we must have $(v_0,v_{k+1})\in A(D)$ and by (1.2), we have $\ell=1$.
Then $(v_0,v_{k+1})$ and $v_0Cv_{k+1}$ form a 2-block cycle $c(k,\ell)$ (where $\ell=1$).
This completes the proof of Claim~1.
\end{proof}


As $k+\ell\ge 3$, from now on we may assume that $\ell\ge 2$.
We choose two vertices $u$ and $v$ such that $uv \in E(G)\setminus E(C)$ and the length of $uCv$ is as small as possible but at least $k+1$.
Note that such an edge $uv$ exists as $\delta(G)\ge k+\ell\ge k+2$.
We may assume that $u=v_0$ and $v=v_{r}$ where $n-1>r\ge k+1$.
By Claim 1, it follows that $r\ge k+2\ge 3$ (See Figure~\ref{fig:claim2} (i)).
We also note that by the minimality of $r$, $v_r$ has at most $k+1$ neighbors in $V(v_0 C v_{r})$ and so it has at least $\ell-1$ neighbors in $V(v_{r+1} C v_{n-1})$ and so $\ell-1\le n-r-1$ and so $r+\ell-1 \le n-1$. Thus, $v_0$, $v_1$, \ldots, $v_{r+\ell-2}$, $v_{r+\ell-1}$ are all distinct vertices in $G$.

\bigskip

{\bf Claim 2:} $N_G[v_{r-1}]= V(v_{r-k-1} C v_{r+\ell-1})$ and $N_G[v_{r-2}]=V(v_{r-k-2} C v_{r+\ell-2})$.

\begin{proof}[Proof of Claim 2.]
By the minimality of $r$, we observe that $v_{r-1}$ has no neighbors in $V(v_0Cv_{r-k-2})$.
If $v_{r-1}$ has a neighbor $w$ in $V(v_{r+\ell}C v_{n-1})$, then
by Lemma~\ref{lem:crossing}, the edges $v_0v_r, v_{r-1}w\in E(G)\setminus E(C)$
force a 2-block cycle $c(k,\ell)$ in $D$, a contradiction.
Therefore, $N_G(v_{r-1})\subseteq V(v_{r-k-1} C v_{r+\ell-1})$.
As $v_{r-1}$ has at least $k+\ell$ neighbors in $G$, it follows that $N_G[v_{r-1}]= V(v_{r-k-1} C v_{r+\ell-1})$.

Consider $v_{r-2}$.
If $v_{r-2}$ has a neighbor $w$ in $V(v_{r+\ell}Cv_{n-1})$,
then by Lemma~\ref{lem:crossing}, the edges $v_0v_r, v_{r-2}w\in E(G)\setminus E(C)$ force a 2-block cycle $c(k,\ell)$, a contradiction.
Moreover, by Claim~1, $v_{r+\ell-1}$ is not a neighbor of $v_{r-2}$.
Therefore, we conclude that $N_G(v_{r-2})$ is a subset of $V(v_0 C v_{r+\ell-2})$.
If $r=k+2$, then $v_0=v_{r-k-2}$ and so
$N_G(v_{r-2})\subseteq V(v_{r-k-2} C v_{r+\ell-2})$.
When $r\ge k+3$, by the minimality of $r$, one can observe that  $v_{r-2}$ has no neighbors in $V(v_0Cv_{r-k-3})$, which implies that $N_G(v_{r-2})\subseteq V(v_{r-k-2} C v_{r+\ell-2})$.
As $|N_G(v_{r-2})|\ge k+\ell$, it follows that $N_G[v_{r-2}]=V(v_{r-k-2} C v_{r+\ell-2})$.
This proves Claim~2.
\end{proof}


\begin{figure}\centering
\begin{subfigure}{.25\textwidth}\centering
 \begin{tikzpicture} [scale=0.7, >=triangle 45,mydeco/.style = {decoration = {markings,mark = at position #1 with {\arrow{<}}}} ]
     \path (0:0) coordinate (0)  (240:2cm) coordinate (w)
           (110:2cm) coordinate (vr-k) (130:2cm) coordinate (vr-k-)
           (170:2cm) coordinate (v0)  (70:2cm) coordinate (vr--)
           (50:2cm) coordinate (vr-)(30:2cm) coordinate (vr)
           (290:2cm) coordinate (vr+l)  (310:2cm) coordinate (vr+l-);
     \draw (0) circle (2cm);
     \draw[-] (v0) to (vr);
       \fill (v0) node[left]{\footnotesize$v_0$}
             (w) node[left]{\footnotesize$w$}
             (vr--) node[above]{\footnotesize$v_{r-2}$}
             (vr-) node[right]{\footnotesize$v_{r-1}$}
             (vr) node[right]{\footnotesize$v_{r}$ \tiny $(r\ge k+2)$}
             (vr-k)node[above]{\footnotesize$v_{r-k-1}$}
             (vr-k-)node[left]{\footnotesize$v_{r-k-2}$}
             (vr+l) node[below]{\footnotesize$v_{r+\ell-1}$}
             (vr+l-) node[right]{\footnotesize$v_{r+\ell-2}$} ;
        \fill (v0) circle (3pt)  (vr-k) circle (3pt) (w) circle (3pt)
              (vr-k-) circle (3pt)  (vr--) circle (3pt)
              (vr-) circle (3pt) (vr) circle (3pt)
              (vr+l) circle (3pt)  (vr+l-) circle (3pt);
    \draw[postaction = {mydeco=-0.4 ,decorate},
             postaction = {mydeco=0 ,decorate}, postaction = {mydeco=0.2 ,decorate}] (0:0) circle (2cm);
        \fill (-90:3cm) node[below]{(i)};
\end{tikzpicture}
\end{subfigure}  \qquad    \quad
\begin{subfigure}{.25\textwidth}\centering
 \begin{tikzpicture} [scale=0.7, >=triangle 45,mydeco/.style = {decoration = {markings,mark = at position #1 with {\arrow{<}}}} ]
     \path (0:0) coordinate (0)
           (110:2cm) coordinate (vr-k) (130:2cm) coordinate (vr-k-)
           (170:2cm) coordinate (v0)  (70:2cm) coordinate (vr--)
           (50:2cm) coordinate (vr-)(30:2cm) coordinate (vr)
           (290:2cm) coordinate (vr+l)  (310:2cm) coordinate (vr+l-);
     \draw (0) circle (2cm);
     \draw[->] (vr-k-) to (vr--); \draw[->] (vr--) to (vr+l-);
     \draw[->] (vr-k) to (vr-); \draw[->] (vr-) to (vr+l);
       \fill (v0) node[left]{\footnotesize$v_0$}
             (vr--) node[above]{\footnotesize$x$}
             (vr-) node[right]{\footnotesize$v$}
             (vr) node[right]{\footnotesize$v_{r}$ \tiny $(r\ge k+2)$}
             (vr-k)node[above]{\footnotesize$u_1$}
             (vr-k-)node[left]{\footnotesize$y_1$}
             (vr+l) node[below]{\footnotesize$u_2$}
             (vr+l-) node[right]{\footnotesize$y_2$} ;
        \fill (v0) circle (3pt)  (vr-k) circle (3pt)
              (vr-k-) circle (3pt)  (vr--) circle (3pt)
              (vr-) circle (3pt) (vr) circle (3pt)
              (vr+l) circle (3pt)  (vr+l-) circle (3pt);
    \draw[postaction = {mydeco=-0.4 ,decorate},
             postaction = {mydeco=0 ,decorate}, postaction = {mydeco=0.2 ,decorate}] (0:0) circle (2cm);
         \fill (-90:3cm) node[below]{(ii)};
\end{tikzpicture}
\end{subfigure}
\caption{Figures for Lemma~\ref{lem:Ham-delta}}\label{fig:claim2}
\end{figure}
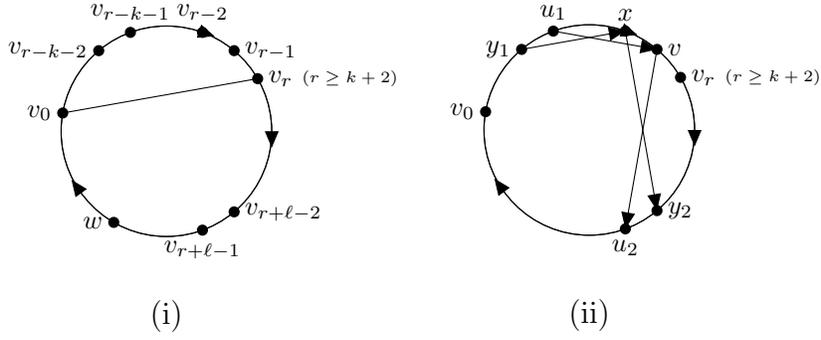

We are ready to arrive at the final contradiction.
For simplicity, let (see Figure~\ref{fig:claim2} (ii))
\[y_1=v_{r-k-2},~~ u_1=v_{r-k-1}, ~~ x=v_{r-2}, ~~v=v_{r-1}, ~~y_2=v_{r+\ell-2},~~ u_2=v_{r+\ell-1}.\]
By Claim~2, $vu_2, xy_2\in E(G)$. Note that $vu_2$ and $xy_2$ are edges not in $E(C)$, since $n\ge \ell+2$.
Since $|vCy_2|=\ell-1$ and $|u_2Cx|=n-\ell-1\ge k$, by Lemma~\ref{lem:crossing},
we have $(v,u_2), (x,y_2)\in A(D)$.
If $k=1$, then the arc $(v,u_2)$ and the path $vCu_2$ form a 2-block cycle $c(k,\ell)$ in $D$.
Thus $k\ge 2$.

By Claim~2 again, $u_1v,xy_1\in E(G)$. Since $k\ge 2$ and so $n\ge k+2$, $u_1v$ and $xy_1$ are edges not in $E(C)$.
As $|u_1Cx|=k-1$ and $|vCy_1|=n-k-1\ge \ell$, by Lemma~\ref{lem:crossing}, we have $(u_1,v), (y_1,x)\in A(D)$. Then the two paths $u_1Cx + (x,y_2)$ and $(u_1,v)+vCy_2$ have the length $k$ and $\ell$, respectively,
and so they induce a 2-block cycle $c(k,\ell)$ in $D$. We reach a contradiction and this completes the proof of Lemma~\ref{lem:Ham-delta}.
\end{proof}


Now we can derive Theorem~\ref{thm:Hamiltonian} from Lemma~\ref{lem:Ham-delta}.

\begin{proof}[Proof of Theorem~\ref{thm:Hamiltonian}.]
It suffices to show that any Hamiltonian digraph $D$ with no $c(k,\ell)$ is $(k+\ell-1)$-degenerate.
Suppose the above statement is false, and let $D$ be a counterexample to it with minimum number of vertices.
That is, $D$ is an $n$-vertex Hamiltonian digraph with no $c(k,\ell)$ and is not $(k+\ell-1)$-degenerate,
but any Hamiltonian digraph with no $c(k,\ell)$ and with less than $n$ vertices is $(k+\ell-1)$-degenerate.
Let $G$ be the underlying graph of $D$ and $C=v_0, v_1, \ldots , v_{n-1}, v_0$ a Hamiltonian cycle of $D$.

If $\delta(G)\ge k+\ell$, then by Lemma~\ref{lem:Ham-delta}, $D$ contains a 2-block cycle $c(k,\ell)$, a contradiction.
Thus there exists some vertex in $G$ which has neighbors less than $k+\ell$. We may assume $|N_G(v_0)|\le k+\ell-1$.
Let $D'$ be the digraph obtained from $D$ by deleting $v_0$ and adding an arc $(v_{n-1},v_1)$,
and $G'$ be the underlying graph of $D'$. Clearly $D'$ is Hamiltonian, $G'$ is an underlying graph of $D'$, and $G-v_0$ is a subgraph of $ G'$.
Suppose $D'$ contains a 2-block cycle $c(k,\ell)$, say $H'$. If the arc $(v_{n-1},v_1)$ is not in $H'$, then $H'$ is a subgraph of $D$,
which yields a contradiction. Thus $H'$ does use $(v_{n-1},v_1)$.
However, the digraph obtained from $H'$ by replacing the arc $(v_{n-1},v_1)$ with the path $v_{n-1},v_0, v_1$ of length two is a 2-block cycle $c(k,\ell)$ as well,
which is contained in $D$, a contradiction. Hence, the Hamiltonian digraph $D'$ contains no 2-block cycle $c(k,\ell)$.
Then by our hypothesis, $D'$ is $(k+\ell-1)$-degenerate, so is $G'$.
Since $G-v_0$ is a subgraph of $G'$ and $|N_G(v_0)|\le k+\ell-1$, it follows that $G$ (and thus $D$) is also $(k+\ell-1)$-degenerate.
This completes the proof of Theorem~\ref{thm:Hamiltonian}
\end{proof}


\section{Proof of Theorem~\ref{thm:main:general}}
In this section, we will prove  Theorem~\ref{thm:main:general} for general strong digraphs.
The plan is first to reduce the upper bound of $\chi(D)$ to $\chi(F)$ for some special subdigraphs $F$ of a strong digraph $D$ (this part will be done in subsection~\ref{subsec:reduction}); and then we study some structural properties on $F$ in subsection~\ref{ssec:cycle-tree}; and finally in subsection~\ref{ssec:cycle-tree:coloring}, we obtain the upper bound for $\chi(F)$ and complete the proof.

\subsection{Some notations}\label{subsec:notation}
In this subsection, we introduce some notations and terminologies which will play important roles in the coming proofs.
Let $D$ be a digraph and $S\subset V(D)$. Let $D/S$ denote the {\it contraction} of $S$ in $D$,
i.e., a digraph obtained by contracting $S$ into a new vertex $v_S$ and adding arcs $(x,y)$ of the following two kinds:
\begin{itemize}
\item[(a)] $x=v_S$ and $y \in V(D)\setminus S$, if there exists $(w,y)\in A(D)$ for some $w\in S$, and
\item[(b)] $x\in V(D)\setminus S$ and $y =v_S$, if there exists $(x,w)\in A(D)$ for some $w\in S$.
\end{itemize}
For a vertex $v\in V(D/S)$, the \textit{preimage} $\varphi(v)$ of $v$ is defined by
\[\varphi(v) =\begin{cases} S & \text{if }v=v_S\\
\{v\} &\text{otherwise.}\end{cases}\]
And for $B\subset V(D/S)$, the \textit{preimage} $\varphi(B)$ of $B$ is defined to be $\varphi(B)=\bigcup_{v\in B} \varphi(v).$

We say a strong digraph $\mathcal{T}$ is a \textit{cycle-tree} if there is an ordering $C_1,C_2,\ldots,C_m$ of all cycles in $\mathcal{T}$ such that for $2\le i \le m$,
\begin{eqnarray*} &&\left| \ V(C_i) \cap \left(\cup_{j\in[i-1]} V(C_j)\right)\ \right| =1.\end{eqnarray*}
We also say this ordering $C_1,C_2,\ldots,C_n$ is a \textit{cycle-tree ordering} of $\mathcal{T}$.
See Figure~\ref{fig:cycle-tree} for an illustration of a cycle-tree.
A cycle-tree $\mathcal{T}$ is called a \textit{cycle-path}, if there exists an ordering $C_0$, $C_1$, $\ldots$, $C_m$ of all cycles of $\mathcal{T}$
such that $|V(C_{i})\cap V(C_{j})|\le 1$ for any distinct $i,j\in\{0,1,\ldots, m\}$,
and the equality holds if and only if $|i-j|=1$.
Here, the cycles $C_0$ and $C_m$ are called the {\it end-cycles} of $\mathcal{T}$, and the \textit{length} of such a cycle-path is defined to be $m$.

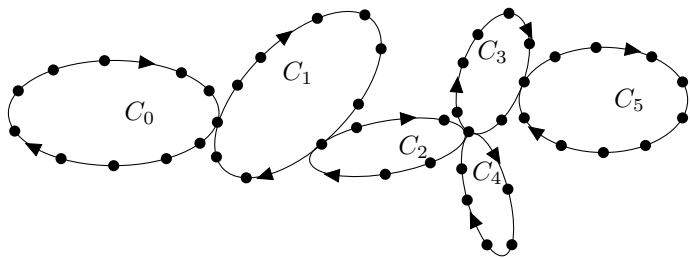
\begin{figure}[h!]\centering
\centering
 \begin{tikzpicture} [scale=0.7, >=triangle 45,mydeco/.style = {decoration = {markings,mark = at position #1 with {\arrow{<}}}} ]
     \path (0,0.65)  coordinate (0)
           (9.3,0.9) coordinate (5)
           (3.5,1) coordinate (1)
(5.2,0) coordinate (2)
(7.1,-0.9) coordinate (4)
(7.2,1.4) coordinate (3);
\fill (0) node[right]{\footnotesize$C_0$}
(1) node[above]{\footnotesize$C_1$}
(2) node[right]{\footnotesize$C_2$}
(3) node[above]{\footnotesize$C_3$}
(4) node[above]{\footnotesize$C_4$}
(5) node[right]{\footnotesize$C_5$};
\draw[rotate=45, postaction = {mydeco=-0.4 ,decorate}, postaction = {mydeco=0.2 ,decorate}]  (1) ellipse (2 and 1);
\draw[rotate=10, postaction = {mydeco=-0.4 ,decorate}, postaction = {mydeco=0.2 ,decorate}]  (2) ellipse (1.5 and 0.5);
\draw[rotate=-20, postaction = {mydeco=-0.4 ,decorate}, postaction = {mydeco=0.2 ,decorate}]   (3) ellipse (0.6 and 1.2);
\draw[rotate=15, postaction = {mydeco=-0.4 ,decorate}, postaction = {mydeco=0.15 ,decorate}]   (4) ellipse (0.4 and 1.2);
\draw[postaction = {mydeco=-0.4 ,decorate}, postaction = {mydeco=0.2 ,decorate}]  (0) ellipse (2 and 1);
\draw[postaction = {mydeco=-0.4 ,decorate}, postaction = {mydeco=0.2 ,decorate}]    (5) ellipse (1.6 and 1);
\fill (0)+(95:2 and 1) circle (3pt)
(0)+(125:2 and 1) circle (3pt)
(0)+(155:2 and 1) circle (3pt)
(0)+(200:2 and 1) circle (3pt)
(0)+(240:2 and 1) circle (3pt)
(0)+(270:2 and 1) circle (3pt)
(0)+(300:2 and 1) circle (3pt)
(0)+(325:2 and 1) circle (3pt)
(0)+(350:2 and 1) circle (3pt)
(0)+(50:2 and 1) circle (3pt)
(0)+(25:2 and 1) circle (3pt);
\fill (5)+(95:1.6 and 1) circle (3pt)
(5)+(125:1.6 and 1) circle (3pt)
(5)+(160:1.6 and 1) circle (3pt)
(5)+(200:1.6 and 1) circle (3pt)
(5)+(240:1.6 and 1) circle (3pt)
(5)+(270:1.6 and 1) circle (3pt)
(5)+(300:1.6 and 1) circle (3pt)
(5)+(340:1.6 and 1) circle (3pt)
(5)+(20:1.6 and 1) circle (3pt)
(5)+(55:1.6 and 1) circle (3pt);

\fill[rotate=-20]
(3)+(100:0.6 and 1.2) circle (3pt)
(3)+(150:0.6 and 1.2) circle (3pt)
(3)+(180:0.6 and 1.2) circle (3pt)
(3)+(230:0.6 and 1.2) circle (3pt)
(3)+(265:0.6 and 1.2) circle (3pt)
(3)+(320:0.6 and 1.2) circle (3pt)
(3)+(40:0.6 and 1.2) circle (3pt);
\fill[rotate=10]
(2)+(110:1.5 and 0.5) circle (3pt)
(2)+(265:1.5 and 0.5) circle (3pt)
(2)+(300:1.5 and 0.5) circle (3pt)
(2)+(40:1.5 and 0.5) circle (3pt);
\fill[rotate=45] (1)+(90:2 and 1) circle (3pt)
(1)+(165:2 and 1) circle (3pt)
(1)+(205:2 and 1) circle (3pt)
(1)+(260:2 and 1) circle (3pt)
(1)+(290:2 and 1) circle (3pt)
(1)+(330:2 and 1) circle (3pt)
(1)+(110:2 and 1) circle (3pt)
(1)+(50:2 and 1) circle (3pt)
(1)+(10:2 and 1) circle (3pt);
\fill[rotate=15]
(4)+(150:0.4 and 1.2) circle (3pt)
(4)+(180:0.4 and 1.2) circle (3pt)
(4)+(230:0.4 and 1.2) circle (3pt)
(4)+(300:0.4 and 1.2) circle (3pt)
(4)+(360:0.4 and 1.2) circle (3pt);
\end{tikzpicture}
\caption{A cycle-tree with a cycle-tree ordering $C_0, C_1,..., C_5$.}\label{fig:cycle-tree}
\end{figure}

Cycle-trees have some `tree-like' properties.
For two distinct cycles $C, C'$ of a cycle-tree $\mathcal{T}$, there exists a uniquely determined cycle-path in $\mathcal{T}$ with end-cycles $C$ and $C'$,
which we denote by $\Lambda_{\mathcal{T}}(C,C')$.
Moreover, for any two vertices $u,v\in V(\mathcal{T})$, there also exist a unique path from $u$ to $v$
and a unique path from $v$ to $u$ in $\mathcal{T}$, which we
denote by $u\mathcal{T}v$ and $v\mathcal{T}u$, respectively.

\subsection{Reducing $\chi(D)$ to $\chi(F)$}\label{subsec:reduction}

Let $k$ and $\ell$ be integers such that $k\ge \ell\ge 1$ and $k\ge 2$, and
let $D$ be a strong digraph with no 2-block cycle $c(k,\ell)$.
In this subsection, we will partition $V(D)$ so that each part induces a certain subdigraph $F$ and
show that $\chi(D)$ can be bounded from above by the maximum $\chi(F)$ (see Lemma~\ref{lem:chi(D)}).

We first define a sequence of strong digraphs $D^{(0)}$, $D^{(1)}$, $\ldots$, $D^{(m)}$ and
a sequence of cycles $C^{(0)}$, $C^{(1)}$, $\ldots$, $C^{(m-1)}$ as following.
Initially, let $D^{(0)}=D$. Now suppose that $D^{(i)}$ has been defined.
If $\chi(D^{(i)})\ge 2k-2$, then in view of Theorem~\ref{thm:bondy},
$D^{(i)}$ has at least one cycle of length at least $2k-2$.
Let $C^{(i)}$ be a longest cycle in $D^{(i)}$ and let $D^{(i+1)}=D^{(i)}/V(C^{(i)})$, which is the digraph obtained from $D^{(i)}$ by contracting $V(C^{(i)})$. Otherwise, $\chi(D^{(i)})\le 2k-3$ and we then stop. This procedure is well-defined since for a strong digraph, a contraction of a set of vertices which induces a strong subdigraph is also strong.

We emphasize that the above definition will be fundamental and we will constantly refer to it in the coming proofs.
Let us collect some properties on $D^{(j)}$ and $C^{(j)}$.
It is clear that $\chi(D^{(m)})\le 2k-3$, and $C^{(i)}$ is a subgraph of $D^{(i)}$ for each $i\in \{0,1,..., m-1\}$.
Denote the collection of the lengths of $C^{(i)}$'s by \[{L}:=\{|C^{(0)}|, |C^{(1)}|,\ldots, |C^{(m-1)}|\}.\]
The following is also easily obtained from the fact that each $C^{(j)}$ is chosen to be a longest one in $D^{(j)}$. We omit the proof.
\begin{proposition}\label{prop:C^i}
$|C^{(0)}|\ge |C^{(1)}|\ge \ldots \ge |C^{(m-1)}|\ge 2k-2$.
\end{proposition}

For each $j\in \{0,1,...,m\}$, $D^{(j)}$ is also a strong digraph with no 2-block cycle $c(k,\ell)$.

\begin{proposition}\label{prop:D^i}
For each $j\in \{0,1,...,m\}$, $D^{(j)}$ contains no $c(k,\ell)$.
\end{proposition}

\begin{proof}
We prove by induction on $j$. The base case $j=0$ follows as $D$ contains no $c(k,\ell)$.
Suppose that $D^{(j)}$ contains no $c(k,\ell)$.
If $D^{(j+1)}=D^{(j)}/C^{(j)}$ contains a 2-block cycle $c(k,\ell)$ (call this subdigraph $H$),
then it is straightforward to see that the subdigraph of $D^{(j)}$
obtained from $H$ by un-contracting $C^{(j)}$ also contains a 2-block cycle $c(k,\ell)$, a contradiction.
This proves the proposition.
\end{proof}

For each $v\in V(D^{(m)})$ and $j\in [m]$, we recursively define the $j^{th}$ {\it preimage} $\varphi^{(j)}(v)$ of $v$ as following:
let $\varphi^{(1)}(v)=\varphi(v)\subseteq V(D^{(m-1)})$ and for $2\le j\le m$,
  \[\varphi^{(j)}(v) = \varphi\left(\varphi^{(j-1)}(v)\right)\subseteq V(D^{(m-j)}).\]
In particular, $\varphi^{(m)}(v)\subseteq V(D)$.
So $\varphi^{(m)}(v)$ for all $v\in V(D^{(m)})$ form a partition of $V(D)$.

We are ready to prove the main lemma of this subsection.

\begin{lemma}\label{lem:chi(D)}
It holds that
\[\chi(D) \le (2k-3)\times  \max_{v\in V(D^{(m)})} \chi(D[\varphi^{(m)}(v)]).\]
\end{lemma}

\begin{proof}
Let $t:=\max \chi(D[\varphi^{(m)}(v)])$ over all $v\in V(D^{(m)})$.
We recall that $\chi(D^{(m)})\le 2k-3$. So $V(D^{(m)})$ can be partitioned into $2k-3$ independent sets,
say $B_1, B_2, \ldots, B_{2k-3}$. (Here it is possible that $B_i=\emptyset$ for some $i\in [2k-3]$.)
For each $i\in [2k-3]$, let $V_i$ be the union of $\varphi^{(m)}(v)$ over all $v\in B_i$.
So $V_1,...,V_{2k-3}$ form a partition of $V(D)$.
Since $B_i$ is independent in $D^{(m)}$, it is easy to see that for distinct vertices $u$ and $v$ of $B_i$, there are no arcs between $\varphi^{(m)}(u)$ and $\varphi^{(m)}(v)$ in $D$.
This shows that for each $i\in [2k-3]$, the chromatic number of each induced subdigraph $D[V_i]$ is at most $t$, implying that $\chi(D)\le (2k-3)\cdot t$.
This completes the proof of Lemma~\ref{lem:chi(D)}
\end{proof}

In the rest of this section, we denote by $F:=D[\varphi^{(m)}(s)]$ for an arbitrary vertex $s\in V(D^{(m)})$.

\subsection{Properties on $F$}\label{ssec:cycle-tree}
In the following two lemmas we obtain some useful properties on $F$. Recall that
$L$ is the set of lengths of cycles $C^{(i)}$'s.

\begin{lemma}\label{lem:cycle-tree}
If $F$ has more than one vertex, then $F$ contains a cycle-tree $\mathcal{T}$ as a spanning subdigraph such that
the length of every cycle of $\mathcal{T}$ is from ${L}$ (and thus at least $2k-2$).\end{lemma}

\begin{proof}
Recall that for each $j\in [m]$, the $j^{th}$ preimage $\varphi^{(j)}(v)$ is a subset of $V(D^{(m-j)})$.
Let \[D_j:=D^{(m-j)}[\varphi^{(j)}(v)].\]
We prove by induction on $j\in [m]$ that every $D_j$ either consists of $\{v\}$, or contains a cycle-tree $\mathcal{T}_j$ as a spanning subdigraph such that the length of every cycle of $\mathcal{T}$ is from ${L}$.
This is clearly sufficient, as $F=D_m$.
The base case $j=1$ is trivial, as by definition, $\varphi(v)$ is either $\{v\}$
or the cycle $C^{(m-1)}$ which is a spanning cycle-tree of $D_1$.

Now suppose that the statement holds for some $j\in [m]$ and we consider $D_{j+1}$.
If $D_j$ consists of $\{v\}$ (i.e., $\varphi^{(j)}(v)=\{v\}$),
then $\varphi^{(j+1)}(v)=\{v\}$ or $C^{(m-j-1)}$, and similarly as the base case,
we see that the statement also holds for $D_{j+1}$.
Hence, we may assume that $D_j$ contains a cycle-tree $\mathcal{T}_j$ as a spanning subdigraph
such that the length of every cycle of $\mathcal{T}_j$ is from ${L}$.

We point out that $D_j$ and  $D_{j+1}$ are induced subgraphs of  $D^{(m-j)}$ and $D^{(m-j-1)}$, respectively,
and $D^{(m-j)}=D^{(m-j-1)}/C^{(m-j-1)}$, where the new vertex of $D^{(m-j)}$, say $u$,
is obtained by contracting the cycle $C:=C^{(m-j-1)}$.
If $\varphi^{(j+1)}(v)=\varphi^{(j)}(v)$, then clearly $D_{j+1}=D_j$ and we are done.
So we may assume that $\varphi^{(j+1)}(v)\neq \varphi^{(j)}(v)$.
Then it must be the case that $u\in \varphi^{(j)}(v)=V(D_j)$ and thus $V(D_{j+1})=(V(D_j)\setminus \{u\})\cup V(C)$.
In fact, we also have $D_j=D_{j+1}/C$.

Since $u$ is a vertex of $D_j$ and $D_j$ contains a spanning cycle-tree $\mathcal{T}_j$ with the described property,
there exists a non-empty set $\mathcal{C}$ of cycles in $\mathcal{T}_j$
containing the vertex $u$.
Take $C'$ to be any cycle in $\mathcal{C}$. Then there exist $x',y'\in V(C')\setminus \{u\}$ such that
$(x',u)$ and $(u,y')$ are the two arcs of $C'$ incident to $u$.
Un-contracting back to $D_{j+1}$, we see there are two arcs $(x',x)$ and $(y,y')$ of $D_{j+1}$ for some $x,y\in V(C)$.

We claim that $x=y$ for every such $C'\in \mathcal{C}$. Suppose for a contradiction that $x\neq y$.
Then there exists a path $P:=(y,y')+ y' C' x' + (x',x)$ in $D_{j+1}$ from $y$ to $x$ such that
\[|P|\ge 1+(|C'|-2)+1=|C'|\ge 2k-2 \ge k,\]
where the last inequality is from $k\ge 2$.
If $|yCx|\ge \ell$, then the paths $P$ and $yCx$ are internally disjoint paths in $D_{j+1}$ from $y$ to $x$ of length at least $k$ and $\ell$, respectively, and thus they form a 2-block cycle $c(k,\ell)$ in $D_{j+1}$ (and thus in $D^{(m-j-1)}$), a contradiction to Proposition~\ref{prop:D^i}.
So we have $|yCx|\le \ell-1$. Then $Q:=P+ xCy$ is a cycle in $D_{j+1}$ such that (note $k\ge \ell$ and $k\ge 2$)
\[|Q|=|P|+|C|-|yCx|\ge (2k-2)+|C|-(\ell-1)>|C|,\]
contradicting the fact that $C=C^{(m-j-1)}$ is a longest cycle in $D^{(m-j-1)}$.
This proves $x=y$.

For any $C'\in \mathcal{C}$, we update $C'$ to be a cycle $C''$ in $D_{j+1}$ by replacing $(x',u), (u,y')$ with $(x',x), (x,y')$.
Clearly, $|C''|=|C'|\in {L}$. We then can define a subdigraph $\mathcal{T}_{j+1}$
of $D_{j+1}$ to be obtained from $\mathcal{T}_j$ by replacing all cycles $C'\in {C}$
with the corresponding $C''$ and by adding the new cycle $C^{(m-j-1)}$.
It is easy to check that $\mathcal{T}_{j+1}$ indeed is a spanning cycle-tree of $D_{j+1}$,
and the length of every cycle of $\mathcal{T}_{j+1}$ is from ${L}$. This finishes the proof of Lemma~\ref{lem:cycle-tree}.
\end{proof}


\noindent{\bf Remark.} From the proof of Lemma~\ref{lem:cycle-tree},
we also see that during the contraction process (say from $D^{(m-j-1)}$ to $D^{(m-j)}$ by contracting the cycle $C=C^{(m-j-1)}$),
either the spanning cycle-tree $\mathcal{T}_{j+1}$ of $D_{j+1}$ remains unchanged,
or $C$ is a cycle in $\mathcal{T}_{j+1}$ and $\mathcal{T}_j=\mathcal{T}_{j+1}/C$.
Hence, if we look at the whole contraction process,
at each step the spanning cycle-tree $\mathcal{T}$ of $F$ will either remain the same or contract one of its cycles.

\medskip

From now on, we assume that $F$ has at least two vertices and let $\mathcal{T}$ be a fixed cycle-tree of $F$ guaranteed in Lemma~\ref{lem:cycle-tree}.
An arc $(x,y)$ in $F$ is called an {\it external arc} (with respect to $\mathcal{T}$),
if there is no cycle of $\mathcal{T}$ containing both $x$ and $y$.
The following lemmas tells that
for an external arc $(x,y)$ in $F$, the path $y\mathcal{T}x$ must be short.

\begin{lemma}\label{lem:short-backward-path}
Suppose that  $F$ has an external arc $(x,y)$.
Let $C_x$ and $C_y$ be the cycles of $\mathcal{T}$  containing $x$ and $y$, respectively, such that $|\Lambda_{\mathcal{T}}(C_x,C_y)|$ is the minimum.
Let $u$ and $v$ be the common vertices of the first and the last two cycles of the cycle-path $\Lambda_{\mathcal{T}}(C_x,C_y)$, respectively.
Then both $|v\mathcal{T}x|$ and $|y\mathcal{T}u|$ are at most $\ell-2$.
\end{lemma}

\begin{proof}
Let $\Lambda_{\mathcal{T}}(C_x,C_y): C_0, C_1,\ldots, C_t$ for some $t\ge 1$.
By Lemma~\ref{lem:cycle-tree}, the length of each $C_i$ is from ${L}$ and at least $2k-2$.
Let $\gamma$ be the maximum of $|C_i|$ over $0\le i\le t$.
Since $t\ge 1$, we have $\sum_{i=0}^{t}|C_i|\ge \gamma + (2k-2).$

Let $j\in [m]$ be the minimum integer such that $|C^{(j)}|=\gamma$.
By the minimality, we point out that all contracting cycles $C^{(i)}$ obtained
before $D^{(j)}$ have lengths strictly bigger than $\gamma$.
By the remark after Lemma~\ref{lem:cycle-tree},
this also shows that all cycles in $\Lambda_{\mathcal{T}}(C_x,C_y)$ have not been contracted before $D^{(j)}$,
and as a result, $D^{(j)}$ contains the induced subdigraph $H$ of $F$ restricted on $\Lambda_{\mathcal{T}}(C_x,C_y)$.

Note that by the choice of $C_x$ and $C_y$, $x\neq u,v$ and $y\neq u,v$.
We first prove that $|v\mathcal{T}x|\le \ell-2$.
Suppose for a contradiction that $|v\mathcal{T}x|\ge \ell-1$.
If $|v\mathcal{T} y|\ge k$, then $v\mathcal{T}x+(x,y)$ and $v\mathcal{T}y$
are internally disjoint paths from $v$ to $y$ of length at least $\ell$ and $k$,
giving a 2-block cycle $c(k,\ell)$, a contradiction.
Hence, $|v\mathcal{T} y|\le k-1$, implying that $|y\mathcal{T}v|\ge \ell-1$ (as $C_t=v\mathcal{T}y+y\mathcal{T}v$).
If $|x\mathcal{T}v|\ge k$, then the paths $(x,y)+y\mathcal{T}v$ and $x\mathcal{T}v$ generate a 2-block cycle $c(k,\ell)$ in $D$, a contradiction.
So $|x\mathcal{T}v|\le k-1$.
Let $C=(x,y)+ y\mathcal{T}x$. Then $C$ is a cycle with length
\begin{eqnarray*}
|C| &=&\sum_{i=0}^{t}|C_i| +1 - |x\mathcal{T}v| - |v\mathcal{T} y| \ge \gamma+(2k-2) +1 -(k-1)-(k-1)>\gamma.
\end{eqnarray*}
Note that $C$ is also a cycle of $H$ and also a cycle of $D^{(j)}$.
However, this is a contradiction, as $C^{(j)}$ has length $\gamma$ and is a longest cycle in $D^{(j)}$.
Therefore, it cannot happen that $|v\mathcal{T}x|\ge \ell-1$, and so  $|v\mathcal{T}x|\le \ell-2$.

We then show that $|y\mathcal{T}u|\le \ell-2$. Suppose not that $|y\mathcal{T}u|\ge \ell-1$.
Then the paths $(x,y)+y\mathcal{T}u$ has length at least $\ell$.
To avoid a 2-block cycle $c(k,\ell)$, we must have $|x\mathcal{T}u|\le k-1$.
But we also have $|u\mathcal{T}x|\le |v\mathcal{T} x|\le \ell-2$.
Therefore $|C_0|=|x\mathcal{T} u|+|u\mathcal{T}x|\le (k-1) + (\ell-2) < 2k-2,$
a contradiction. This completes the proof.
\end{proof}

\noindent{\bf Remark.} From the proof of Lemma~\ref{lem:short-backward-path}, we know that
if $F$ has an external arc $(x,y)$, then it cannot happen that $|v\mathcal{T}x|\ge \ell-1$.
This implies that if $F$ has an external arc, then $\ell \ge 2$.

\subsection{Coloring $F$} \label{ssec:cycle-tree:coloring}

In this subsection, our goal is to find a proper coloring of $F$ using $O(k+\ell)$ colors,
which completes the proof of Theorem~\ref{thm:main:general}.
Recall that $\mathcal{T}$ is the spanning cycle-tree of $F$ (fixed from Lemma~\ref{lem:cycle-tree}).
Fix a cycle-tree ordering $C_0,C_1,\ldots,C_n$ of $\mathcal{T}$.
For simplicity, we write $\Lambda_{\mathcal{T}}(C_i,C_j)$ as $\Lambda(C_i,C_j)$.

We define two spanning subdigraphs $F^{(1)}$ and $F^{(2)}$ of $F$ such that $A(F)=A(F^{(1)})\cup A(F^{(2)})$.
Before processing, we need to introduce some notations.
Let $C$ be a cycle in $\mathcal{T}$ with $C\neq C_0$.
We call the second last cycle in $\Lambda(C_0,C)$ the \textit{parent} of $C$ and denote it as $p(C)$.
The unique vertex $p\in V(C)\cap V(p(C))$ is called the \textit{parent vertex} of $C$.
Note that the notions of the parent cycle and the parent vertex are uniquely defined for all cycles of $\mathcal{T}$ except $C_0$.
For every $v\in V(F)$, let $C_v$ be the cycle of $\mathcal{T}$ containing $v$ which has the shortest cycle-path to the cycle $C_0$.

We then define a function $\phi: V(\mathcal{T})\to \{0,1\}$ by letting
\[
 \phi(v)=\left\{\begin{array}{ll} 1  &\text{ if }C_v\neq C_0 \text{ and }  |v C_v p_v|(=|v\mathcal{T}p_v|)\le \ell-2,\\
 0&\text{ otherwise }\end{array}\right.\]
where $p_v$ is the parent vertex of $C_v$, respectively.
Let $F^{(2)}$ be the spanning subdigraph of $F$ such that
\[A(F^{(2)})=\{ (u,v) \mid (u,v)\text{ is an external arc of }F\text{ and }  \phi(u)\neq \phi(v)\}. \]
Let $F^{(1)}$ be the spanning subdigraph of $F$ such that $A(F^{(1)})=A(F)\setminus A(F^{(2)})$.
Clearly, $\phi$ is a proper coloring of $F^{(2)}$ and thus we have
\begin{align}\label{equ:F^2}
\chi(F^{(2)})\le 2.
\end{align}

To complete the proof of Theorem~\ref{thm:main:general},
it suffices (as we should see later) to show $\chi(F^{(1)})\le k+2\ell-1$.
This will be accomplished in Lemma~\ref{lem:degenerate} which in fact provides a slightly stronger result.
In what follows, we first prove a useful lemma, showing that the external arcs of $F^{(1)}$
satisfy some tree-like property.\footnote{It may help understand the proof if one analogizes this as the property
of the depth-first-search tree.}
For an external arc $(u,v)$ of $F$, we say $(u,v)$ is \textit{comparable} if either $C_u$ is a cycle in $\Lambda(C_v,C_0)$ or $C_v$ is a cycle in $\Lambda(C_u,C_0)$.

\begin{lemma}\label{lem:F^1-comparable}
All external arcs of $F^{(1)}$ are comparable.
\end{lemma}

\begin{proof}
Suppose for a contradiction that $F^{(1)}$ contains an external arc $(u,v)$ which is not comparable.
Note that  $\ell\ge 2$ by a remark after Lemma~\ref{lem:short-backward-path}, and $C_u,C_v, C_0$ are three distinct cycles of $\mathcal{T}$.
Let $p_u$ and $p_v$ be the parent vertices of $C_u$ and $C_v$, respectively.
Also note that as $(u,v)$ is not comparable, $p_u$ and $p_v$ are the common vertices of the first and the last two cycles in $\Lambda(C_u,C_v)$, respectively.

We claim that $\phi(u)=0$ and $\phi(v)=1$.
If $\phi(u)=1$, then $|u C_u p_u|\le \ell-2$ and so
\[|p_v\mathcal{T}u|\ge |p_u\mathcal{T} u|=|p_u C_u u|=|C_u|-|uC_up_u|\ge (k+\ell-2)-(\ell-2)=k>\ell-2,\]  a contradiction to Lemma~\ref{lem:short-backward-path}. Thus $\phi(u)=0$.
If $\phi(v)=0$, then $|v C_v p_v|\ge \ell-1$ and so
\[|v\mathcal{T} p_u|\ge |v\mathcal{T}p_v|=|v C_v p_v|> \ell-2,\]
again a contradiction to Lemma~\ref{lem:short-backward-path}.
Thus $\phi(v)=1$.

Therefore, $\phi(u)\neq \phi(v)$, implying that $(u,v)\notin F^{(1)}$. This completes the proof.
\end{proof}

Recall that we have fixed a cycle-tree ordering $C_0,C_1,\ldots,C_n$ of $\mathcal{T}$.
For every $i\in [n]$, let $p_i$ be the parent vertex of $C_i$.
And for every $i\in \{0,1,...,n\}$, let $F_i$ be the induced subdigraph of $F^{(1)}$
restricted on $V(C_0)\cup V(C_1)\cup ...\cup V(C_i)$.
Note that $F_0=F[V(C_0)]$ and $F_n=F^{(1)}$.

\begin{lemma}\label{lem:F^1-degree}
For every $i\in [n]$ and for any $v\in V(C_i)\setminus \{p_i\}$, the number of external neighbors
\footnote{We say $u$ is an {\it external neighbor} of $v$ in $F$ if $(u,v)$ or $(v,u)$ is an external arc of $F$.}
of $v$ in $F_i$ is at most $\max\{0,\ell-2\}$.
\end{lemma}

\begin{proof} If $F$ has no external arc, then it is trivial. Suppose that $F$ has an external arc. Then $\ell\ge 2$ by the remark after Lemma~\ref{lem:short-backward-path}.
Fix $i\in [n]$ and $v\in V(C_i)\setminus\{p_i\}$.
Let $\Lambda=\Lambda(C_i,C_0)$. By the definition of the cycle-tree ordering,
it follows that all cycles in $\Lambda$ are contained in $V(F_i)$.
In view of Lemma~\ref{lem:F^1-comparable}, we see that all external arcs of $F_i$ between $V(C_i)$ and $V(F_{i-1})$ are those
between $V(C_i)$ and $\Lambda\setminus V(C_i)$.
Therefore, to prove the lemma, it suffices to show that the number of external neighbors of $v$ in $\Lambda$ is at most $\ell-2 $.

Let $S^+(v)$ be the set of vertices $u$ on $\Lambda$ such that $(v,u)$ is an external arc of $F^{(1)}$, and let $S^-(v)$ be the set of vertices $u$ on $\Lambda$ such that $(u,v)$ is an external arc of $F^{(1)}$.
Note that the parent vertex $p_i$ of $C_i$ is the common vertex of the first two cycles in $\Lambda$.
Let $w$ be the common vertex of the last two cycles in $\Lambda$. So $w\in V(C_0)$. Let $w^+$ and $w^-$ be the vertices of $C_0$ such that $|wC_0w^+|=\ell-2$ and $|w^-C_0w|=\ell-2$ (see Figure~\ref{fig:F^1-degree}).

\begin{figure}
\centering
 \begin{tikzpicture} [scale=0.8, >=triangle 45,mydeco/.style = {decoration = {markings,mark = at position #1 with {\arrow{<}}}} ]
     \path (2,0) coordinate (0) (4,0) coordinate (1)  (7,0) coordinate (2)
(9,0) coordinate (3)
(8.4,0.8) coordinate (u)
(11,0) coordinate (4)
(14,0) coordinate (5)
(16,0) coordinate (6)
(3,0) coordinate (p)
(8,0) coordinate (z)
(10,0) coordinate (pu)
(15,0) coordinate (w) (15.4,0.8) coordinate (w+) (15.4,-0.8) coordinate (w-) (1.4,0.8) coordinate (v) ;
     \draw (0) circle (1cm);
     \draw  (1) circle (1cm) ;
     \draw (2) circle (1cm) ;
     \draw (3) circle (1cm);
     \draw (4) circle (1cm);
     \draw (5) circle (1cm);
     \draw (6) circle (1cm);
\fill (0) node[right]{\footnotesize$C_i$}
(3) node[right]{\footnotesize$C_u$}
(6) node[right]{\footnotesize$C_0$}
(v) node[left]{\footnotesize$v$}
(u) node[left]{\footnotesize$u$}
             (w) node[left]{\footnotesize$w$}
             (w-) node[right]{\footnotesize$w^-$}
             (w+) node[right]{\footnotesize$w^+$}
             (p) node[right]{\footnotesize$p_i$}
             (z) node[right]{\footnotesize$z$}
                 (pu) node[right]{\footnotesize$p_u$};
       \fill (w) circle (3pt) (z) circle (3pt)  (pu) circle (3pt)
              (w+) circle (3pt)  (w-) circle (3pt)
             (p) circle (3pt) (v) circle (3pt)  (u) circle (3pt)
             (5.3,0) circle (1pt) (5.5,0) circle (1pt) (5.7,0) circle (1pt)
                     (12.3,0) circle (1pt) (12.5,0) circle (1pt) (12.7,0) circle (1pt);
    \draw[postaction = {mydeco=-0.4 ,decorate}, postaction = {mydeco=0.2 ,decorate}] (0) circle (1cm);
    \draw[postaction = {mydeco=-0.4 ,decorate}, postaction = {mydeco=0.2 ,decorate}] (1) circle (1cm);
    \draw[postaction = {mydeco=-0.4 ,decorate}, postaction = {mydeco=0.2 ,decorate}] (2) circle (1cm);
    \draw[postaction = {mydeco=-0.4 ,decorate}, postaction = {mydeco=0.2 ,decorate}] (3) circle (1cm);
    \draw[postaction = {mydeco=-0.4 ,decorate}, postaction = {mydeco=0.2 ,decorate}] (4) circle (1cm);
    \draw[postaction = {mydeco=-0.4 ,decorate}, postaction = {mydeco=0.2 ,decorate}] (5) circle (1cm);
    \draw[postaction = {mydeco=-0.4 ,decorate},
             postaction = {mydeco=0.45 ,decorate}, postaction = {mydeco=0.2 ,decorate}] (6) circle (1cm);
\end{tikzpicture}
\caption{$\Lambda:=\Lambda(C_i,C_0)$}
\label{fig:F^1-degree}\end{figure}
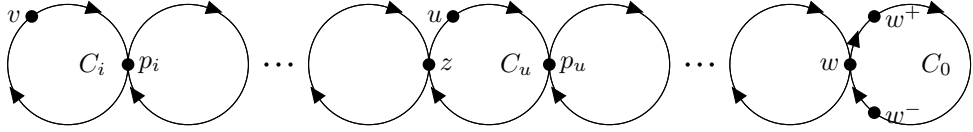

\bigskip

{\bf Claim 1:} If $S^+(v)\neq \emptyset$, then $\phi(v)=0$ and $|S^+(v)|\le \ell-2$.

\begin{proof}[Proof of Claim 1.] Suppose that we take any $u\in S^+(v)$.
We first prove that
\begin{align}\label{equ:S^+}
|u\mathcal{T}p_i|\le \ell-2 \text{  and  } \phi(v)=\phi(u)=0.
\end{align}

\noindent
To see this, note that $\Lambda(C_i,C_u)$ is a subpath of $\Lambda$.
So $p_i$ is the common vertex of the first two cycles in $\Lambda(C_i,C_u)$.
Let $z$ be the common vertex of the last two cycles in $\Lambda(C_i,C_u)$.
By Lemma~\ref{lem:short-backward-path}, we have $|u\mathcal{T} p_i|\le \ell-2$
and $|p_iC_iv|=|p_i\mathcal{T} v|\le |z\mathcal{T}v|\le \ell-2$, implying that $|vC_ip_i|\ge 2k-2-(\ell-2)>\ell-2$ and so $\phi(v)=0$.
Since $(v,u)$ is an external arc in $F^{(1)}$, we have $\phi(u)=\phi(v)=0$.

We then assert that $S^+(v)\subseteq V(w^- \mathcal{T} p_i)$.
Otherwise, there exists some $u\in S^+(v)$ such that $u\not\in V(w^- \mathcal{T} p_i)$.
If $u\in V(C_0)$, then by the definition of $w^-$, we have $|u\mathcal{T} p_i|>|w^-C_0w|=\ell-2$,
a contradiction to \eqref{equ:S^+}.
So $C_u\neq C_0$ and the parent vertex $p_u$ of $C_u$ is well-defined.
To have $u\not\in V(w^-\mathcal{T} p_i)$ and $\phi(u)=0$,
we must have $u\in zC_up_u$ and $|uC_up_u|>\ell-2$, implying that
$|u\mathcal{T} p_i| \ge |u C_up_u| >\ell-2,$
again a contradiction to \eqref{equ:S^+}. This proves the assertion.

Now let $u\in S^+(v)$ be the farthest vertex from $p_i$ in the path $w^-\mathcal{T} p_i$.
Then $S^+(v)\subseteq V(u\mathcal{T}p_i)\setminus\{p_i\}$. By \eqref{equ:S^+},
$|S^+(v)|\le |V(u\mathcal{T}p_i)\setminus \{p_i\}|=|u\mathcal{T}p_i|\le \ell-2$, proving Claim~1.
\end{proof}

\medskip

{\bf Claim 2:} If $S^-(v)\neq \emptyset$, then $\phi(v)=1$ and $|S^-(v)|\le \ell-2$.

\begin{proof}[Proof of Claim 2.]
This will be similar to Claim~1.
We first prove that for any $u\in S^-(v)$,
\begin{align}\label{equ:S^-}
|p_i\mathcal{T}u|\le \ell-2 \text{  and  } \phi(v)=\phi(u)=1.
\end{align}
It is clear that the common vertex of the last two cycles in $\Lambda(C_u,C_i)$ is $p_i$.
Let $z$ be the common vertex of the first two cycles in $\Lambda(C_u,C_i)$.
Then by Lemma~\ref{lem:short-backward-path}, we have $|p_i\mathcal{T} u| \le \ell-2$ and $|v\mathcal{T} p_i|\le |v\mathcal{T}z|\le \ell-2$,
the latter of which implies that $\phi(v)=1$.
Since $(u,v)$ is an external arc in $F^{(1)}$, $\phi(u)=\phi(v)=1$.

Next we show $S^-(v)\subseteq V(p_i\mathcal{T}w^+)$.
Suppose not, then there exists some $u\in S^-(v)$ such that $u\not\in V(p_i\mathcal{T}w^+)$.
If $u\in V(C_0)$ (so $u\in w^+C_0w$), then by the definition of $w^+$, $|p_i\mathcal{T}u|>|wC_0w^+|=\ell-2$,
a contradiction to \eqref{equ:S^-}.
Thus, $C_u\neq C_0$ and the parent vertex $p_u$ is well-defined.
To have $\phi(u)=1$ and $u\not\in p_i\mathcal{T} w^+$,
it must hold that $u\in p_uC_uz$ and $|uC_uz| \le |uC_u p_u| \le \ell-2$,
implying that
\[|p_i\mathcal{T}u|\ge|zC_u u|=|C_u|-|uC_uz|\ge (2k-2)-(\ell-2)>\ell-2,\]
a contradiction to \eqref{equ:S^-}.

Let $u\in S^-(v)$ be the farthest vertex from $p_i$ in the path $p_i\mathcal{T} w^+$.
Then $S^-(v)\subseteq V(p_i\mathcal{T}u)\setminus\{p_i\}$. By~\eqref{equ:S^-},
$|S^-(v)|\le |p_i\mathcal{T}u|\le \ell-2$. This proves Claim~2.
\end{proof}

Claims 1 and 2 also show that at most one of $S^+(v)$ and $S^-(v)$ can be non-empty.
Therefore, since the number of external neighbors of $v$ in $\Lambda$ is $\max\{|S^+(v)|,|S^-(v)|\}$, it is at most $\ell-2$.
We have completed the proof of Lemma~\ref{lem:F^1-degree}.
\end{proof}

\begin{lemma}\label{lem:degenerate}
$F^{(1)}$ is $(k+2\ell-2)$-degenerate.
\end{lemma}

\begin{proof}
If ${F}$ has no external arc, then it is clear that from Theorem~\ref{thm:Hamiltonian}, $F$ is $(k+\ell-1)$-degenerate by considering the cycles $C_n$, \ldots, $C_0$ (the reverse of the cycle-tree ordering) one by one.
In the following, we assume that $F$ has an external arc and so $\ell\ge 2$ by a remark after Lemma~\ref{lem:short-backward-path}.

We prove by induction on $i\in \{0,1,...,n\}$ that each of $F_i$ is $(k+2\ell-2)$-degenerate.
Note that this is sufficient, as $F_n=F^{(1)}$.
The base case $i=0$ follows from Theorem~\ref{thm:Hamiltonian} directly:
since $F_0=F[V(C_0)]$ is Hamiltonian with no 2-block cycle $c(k,\ell)$, $F_0$ is $(k+\ell-1)$-degenerate and thus $(k+2\ell-2)$-degenerate.

Suppose that $F_{i-1}$ is $(k+2\ell-2)$-degenerate. Consider $F_i$, which is the union of $F_{i-1}$, $F[V(C_i)]$ and the external arcs between $V(C_i)$ and $F_{i-1}$.
As $F[V(C_i)]$ is Hamiltonian with no $c(k,\ell)$, by Theorem~\ref{thm:Hamiltonian}, we know $F[V(C_i)]$ is $(k+\ell-1)$-degenerate.
So there exists a linear ordering $v_1$, $\ldots$, $v_t$ of $V(C_i)\setminus\{p_i\}$ such that for any $j\in [t]$, $v_j$ has at most $k+\ell-1$ neighbors in $F[\{v_1,...,v_{j-1}\}]$.
And Lemma~\ref{lem:F^1-degree} says that every vertex in $V(C_i)\setminus\{p_i\}$ has at most $\max\{0,\ell-2\}=\ell-2$ external neighbors in $V(F_{i-1})\setminus\{p_i\}$.
Combining the above, the ordering $v_1$, $\ldots$, $v_t$ of $V(C_i)\setminus\{p_i\}$ also satisfies that for any $j\in [t]$, $v_j$ has at most $(k+\ell-1)+(\ell-2)+1=k+2\ell-2$ neighbors in $F[V(F_{i-1})\cup \{v_1,...,v_{j-1}\}]$.
This, together with that $F_{i-1}$ is $(k+2\ell-2)$-degenerate, implies that $F_i$ is also $(k+2\ell-2)$-degenerate,
finishing the proof of Lemma~\ref{lem:degenerate}.
\end{proof}

Finally, we are ready to prove Theorem~\ref{thm:main:general}.

\begin{proof}[Proof of Theorem~\ref{thm:main:general}.]
Let $D$ be a strong digraph with no 2-block cycle $c(k,\ell)$.
Among all vertices in $V(D^{(m)})$,
choose $v\in V(D^{(m)})$ such that $F:=D[\phi^{(m)}(v)]$ has the maximum $\chi(F)$.
Define $F^{(1)}$ and $F^{(2)}$ as before.
By~\eqref{equ:F^2} and Lemma~\ref{lem:degenerate}, there exist proper colorings
$\rho_1: V(F)\to [k+2\ell-1]$ of $F^{(1)}$ and $\rho_2: V(F)\to \{0,1\}$ of $F^{(2)}$, respectively.
Define $\rho:V(F)\to [k+2\ell-1]\times \{0,1\}$ by letting for every $v\in V(F)$, $\rho(v)=(\rho_1(v),\rho_2(v))$.
Since $A(F)=A(F^{(1)})\cup A(F^{(2)})$, it is easy to verify that $\rho$ is a proper coloring of $F$,
which implies that $\chi(F)\le 2(k+2\ell-1)$.
By Lemma~\ref{lem:chi(D)}, it holds that $\chi(D)\le 2(2k-3)(k+2\ell-1)$.
\end{proof}

It will be interesting to improve the upper bound of Theorem~\ref{thm:main:general} further, for instance, to $O(k+\ell)$.
We direct interested readers to \cite{CHLN} and the survey \cite{Ha} for many related problems.

\end{document}